\documentclass{article}

\usepackage{arxiv}

\usepackage{amsmath, cite, amssymb, amsthm}
\usepackage[utf8]{inputenc} 
\usepackage[T1]{fontenc}    
\usepackage{hyperref}       
\usepackage{url}            
\usepackage{amsfonts}       
\usepackage{graphicx}
\usepackage{xcolor,comment}
\usepackage{indentfirst}
\usepackage{subcaption}
\usepackage{float} 
\numberwithin{equation}{section}

\newtheorem{statement}{Statement}[section]
\newtheorem{proposition}{Proposition}[section]
\newtheorem{remark}{Remark}[section]

\title{Analytical and Reduced-Order Modeling of a Timoshenko Beam with Point Actuation}

\author{
        A.~Zuyev$^{1,2}$ \\
        {\small zuyev@mpi-magdeburg.mpg.de} \\
        {\bf Ju.~Kalosha}$^{2}$ \\
        {\small julykucher@gmail.com} \\
        {\small $^1$Max Planck Institute for Dynamics of Complex Technical Systems} \vspace{-1ex} \\
        {\small Sandtorstrasse~1, Magdeburg, Germany 39106} \\
        {\small $^2$Institute of Applied Mathematics and Mechanics, National Academy of Sciences of Ukraine} \vspace{-1ex} \\
        {\small Shevchenko blvd. 79, Cherkasy, Ukraine 18031} \\
       }

\begin{document}
\newcommand{\dd}[1]{\mathrm{d}#1} 
\newcommand{\MM}{\mathbb M}
\newcommand{\LL}{\mathbb L}
\newcommand{\VV}{\mathbb V}
\newcommand{\WW}{\mathbb W}
\newcommand{\bX}{\mathbf X}
\newcommand{\bY}{\mathbf Y}
\newcommand{\bE}{\mathbf E}
\newcommand{\bA}{\mathbf A}
\newcommand{\bB}{\mathbf B}
\newcommand{\bC}{\mathbf C}
\renewcommand{\arraystretch}{1.4}
\maketitle

\begin{abstract}
    We present frequency-domain models of a controlled mechanical system consisting of a flexible beam and a rigid body. The transfer function for the beam, governed by the Timoshenko equations under the body-beam interface conditions, is derived analytically. Based on this infinite-dimensional representation, reduced-order models are constructed within the framework of Loewner matrices. A comparative analysis of the Bode plots is presented for the beam model using different choices of sensor-actuator pairs to validate the proposed frequency-domain modeling framework.
\end{abstract}



\section{Introduction}\label{sec:Intro}

The study of the motion of complex mechanical systems~--- in particular, industrial structures  comprising elastic and rigid components~--- often relies on mathematical models formulated as infinite-dimensional systems in Hilbert or Banach spaces. The associated representations in the frequency domain encompass transfer functions that are not rational. Therefore, in the context of practical applications, such as parameter identification, model order reduction becomes crucially important~\cite{Antoulas2005, BBF2014, MOD2021, DO2014}.
This creates a demand for constructing reduced-order dynamical models, while taking into account the system's key characteristics in such a way that significant simplification of the model does not cause a substantial loss of accuracy. Consequently, over the past few decades, there has been rising interest in model order reduction methods for dynamical control systems.
There are several fundamental approaches to order reduction, and various iterative methods have been extensively developed.
In the context of studying control systems with an output, one of the most commonly used approaches is based on Loewner matrices~--- it provides a method for constructing a finite-dimensional model based on numerical data derived from the measurement results~\cite{ALI2017, ABG2020, IA2014}.

The present work is focused on the frequency domain modeling of a flexible beam with a lumped rigid attachment that produces a local damping effect.
The governing equations are taken in the form of the Timoshenko beam model with account of local damping.
A hybrid beam-body system was considered in the form of the Euler--Bernoulli beam equation in our previous works~\cite{zuyev2005stabilization,KZ2022, KZB2021}.

In Section~\ref{sec:TimBeamModel}, we obtain the frequency-domain representation of analytical model.
We rely on methods for constructing transfer functions for infinite-dimensional systems that were proposed, e.g., in~\cite{CuMo2009, LA1993}, see also~\cite{B2013, CuZw1995}.
An example of deriving the transfer function for the Euler--Bernoulli beam with structural damping can be found in~\cite{ZG2023}.
Derivation of the reduced-order model followed by the numeric analysis is presented in Section~\ref{sec:rom}.

A comparative frequency-domain analysis of the Euler--Bernoulli and Timoshenko beam models, including numeric verifications, was presented in authors' recent work~\cite{ZK2026PCC}, where the Timoshenko beam frequency-domain representation was obtained in the presence of some limitations on the Laplace parameter.
In the current paper, we develop a generic approach to deriving the transfer function, thus defining it explicitly in its full domain. 

For the obtained output-to-input relation, we present the reduced-order model of the considered system by means of interpolation of the original transfer function at an appropriately chosen sampling set.
The order reduction method is adopted within the scope of Loewner matrices.
The singular value decomposition technique is used to estimate the order of reduction, followed by refinement through the comparison of the relative $H^2$ and $H^\infty$ interpolation errors.

Normally it is expected that the set of all eigenvalues of the system lies entirely in the left half-plane of the complex plane, due to the damping effect.
Although, thorough investigation of the underlined beam-body system revealed two purely imaginary eigenvalues $\pm i\omega_c$, where $\omega_c$ is known as the cutoff, or transition, frequency~--- 
the threshold of the high-frequency range, where the second spectral branch of the Timoshenko beam emerges. At this specific frequency, a unique decoupling occurs, giving rise to a purely rotational mode, while for higher frequencies, the spectrum consists of two coexisting branches with distinct wave propagation characteristics.
The second spectrum was identified for a simply supported (hinged-hinged) Timoshenko beam by Traill-Nash and Collar in 1953~\cite{TNC1953} and by Dolph in 1954~\cite{D1954}, while in earlier work by Goens~\cite{G1931} on determination of the elasticity modulus of a rod that is free at both ends and suspended at its nodes, it is mentioned that the type of oscillation changes above a certain frequency threshold (at frequencies, such that the corresponding wavelengths are equal to the cross-sectional dimensions of the rod).
Over the forthcoming decades, many authors have made an effort in analysing the physical meaning of the second spectrum.
It is known that the first spectrum of the Timoshenko beam agrees with the spectrum of the Euler--Bernoulli beam while the second spectrum is often described as being associated with shear-type vibration rather than bending, and its frequencies are typically much higher than the first spectrum for the same mode number. The eigenmodes for the two branches of eigenfrequencies are fundamentally different~--- at least, this effect is easily detectable for classical simply supported beam model, where the frequency equation (or the dispersion equation in the wave-type representation) is biquadratic which yields decoupling of the roots families and mode representation in trigonometric-hyperbolic form below $\omega_c$ and purely trigonometric above $\omega_c$, see~\cite{ASH2016, Cazzani2016, Majkut2009}.

\section{Analytical Modeling of a Timoshenko Beam with a Rigid Body}\label{sec:TimBeamModel}

We study a mechanical system consisting of a simply supported flexible beam of length $\ell$ and a point attachment in the form of a spring–mass actuator.

Let us align the $x$-axis along the centerline of the beam and define the function $w(x,t)$ whose values correspond to the transversal displacement of the point with coordinate $x$ at time $t$, and the function $\psi(x,t)$ that represents the angle of rotation of the cross-section relative to the principal axis.
We assume the aforementioned functions to be of class $C^0$ w.r.t. $x$ and of class $C^2$ w.r.t. $t$.
Mechanical parameters of the outlined system are as follows:
$\rho_0$ is the volume density~--- the mass per unit volume of the beam, $A$ is the area of the cross section of the beam, $\rho=\rho_0A$ is the linear density~--- the mass per unit length,
$I$ is the cross-sectional inertia moment, $I_\rho=\rho_0I$ is the mass moment of inertia of the cross section;
the Young modulus is denoted by $E$ and the shear modulus, or modulus of rigidity,~--- by $G$; $k$ is the shear correction factor that depends on the geometry of the beam's cross-section, $K=kGA$. The beam's cross-section is assumed to be rectangular, so $k=5/6$, according to~\cite{DS2013, GT1991}.
Besides, the beam is assumed to be homogeneous and prismatic.
The beam is equipped with a shaker~--- a rigid body of mass $m$ attached at a point $\ell_0 \in (0, \ell)$ by a spring of stiffness $\varkappa$; it exerts a point force $F_0$.

An application of the variational principle to the derivation of the Timoshenko beam equations can be found, e.g., in~\cite{ZS2007}. This approach was inherited in~\cite{ZK2026PCC} for the simply-supported beam with shaker described above.
The equations of motion are presented in the following form:
\begin{equation}\label{TB.eqs}
    \begin{array}{cc}
        \begin{aligned}
            \rho \ddot w(x,t) - K (w''(x,t) - \psi'(x,t)) & = 0, \\
            I_\rho \ddot\psi(x,t) - EI \psi''(x,t) - K (w'(x,t) - \psi(x,t)) & = 0,
        \end{aligned} & \quad
        x\neq\ell_0,
    \end{array}
\end{equation}
$w(x,t) \in H^2([0,\ell]\times[t_1,t_2])$, $\psi(x,\cdot) \in H^2[t_1,t_2]$, $\forall x\in[0,\ell]$, $\psi(\cdot,t) \in H^1[0,\ell]$, $\forall t \in [t_1,t_2]$,
with the boundary conditions
\begin{equation}\label{TB.BC}
    w(0,t) = w(\ell,t) = 0, \qquad \psi'(0,t) = \psi'(\ell,t) = 0
\end{equation}
and the interface conditions
\begin{equation}\label{TB.IfC1}
    w(\ell_0-0,t) = w(\ell_0+0,t), \quad
    \psi(\ell_0-0,t) = \psi(\ell_0+0,t), \quad
    \psi'(\ell_0-0,t) = \psi'(\ell_0+0,t),
\end{equation}
\begin{equation}\label{TB.IfC2}
        F_0 - m \ddot w(\ell_0,t) - \varkappa w(\ell_0,t) - d \dot w(\ell_0,t) - K \Big( w'(\ell_0-0,t) - w'(\ell_0+0,t) \Big) = 0.
\end{equation}
Here and further, we denote the spatial derivative by prime and the time derivative by dot.
In~\eqref{TB.IfC2}, the local viscous damping at the point $\ell_0$ is proportional to the transversal velocity of motion with a coefficient $d>0$.

The total energy of the above system can be expressed as
$${\cal W} = \frac12\int\limits_0^\ell \Big( \rho\dot w^2 + I_\rho\dot\psi^2 + EI(\psi')^2 + K(w'-\psi)^2\Big)\dd x + \frac12(m\dot w^2 + \varkappa w^2)\Big|_{x=\ell_0}.$$
Calculating its time derivative along the trajectories of~\eqref{TB.eqs}--\eqref{TB.IfC2} with $F_0=0$, one obtains $\dot{\cal W} = -d(\dot w(\ell_0,t))^2$.
Thus, $\dot{\cal W}$ is a negative semidefinite functional indicating the dissipative property of the suggested damped system.
This fact gives rise to the hypothesis about the Lyapunov stability of classical solutions.

Assume that output signals generated by a sensor attached at some point $\ell_k\in[0,\ell]$ are available.
We consider two types of output: the transversal displacement $y_1(t) = w(\ell_k,t)$ of the point $\ell_k\in[0,\ell]$ of the beam, and the curvature of the beam axis at $\ell_k$.
Detailed description of the curvature models can be found in~\cite{LCR2017} and references therein. In the present paper we consider so called flexural, or mechanical, curvature in the form $y_2(t)=\psi'(\ell_k,t)$.

\subsection{Derivation of the Transfer Function}\label{sec:TF_TimBeam}

For deriving the frequency-domain representation of the dynamical system we assume zero initial conditions and apply the Laplace transform w.r.t. $t$ to input, output and unknown functions in system \eqref{TB.eqs}--\eqref{TB.IfC2}.
Let us denote the Laplace transform operator as ${\cal L}: g(x,t)\mapsto G(x,s)=\int\limits_0^\infty g(x,t) e^{-st}\dd t$, where $s\in{\mathbb C}$ is a frequency parameter.
Denoting
$U(s) = \int\limits_0^\infty F_0(t) e^{-st}\dd t$,\;
$Y(s) = \int\limits_0^\infty y(t) e^{-st}\dd t$,\;
$W(x,s) = \int\limits_0^\infty w(x,t) e^{-st}\dd t$,\;
$\Psi(x,s) = \int\limits_0^\infty \psi(x,t) e^{-st}\dd t$ with $s\in{\mathbb C}$, we obtain the following system:
\begin{equation}\label{LD.sys}
    \begin{aligned}
        & K \big(W''(x,s) - \Psi'(x,s) \big) - \rho s^2\, W(x,s) = 0, \\
        & EI\,\Psi''(x,s) + K \big(W'(x,s) - \Psi(x,s) \big) - I_\rho s^2\, \Psi(x,s) = 0,
    \end{aligned}
\end{equation}
\begin{equation}\label{LD.BC}
    W(0,s) = W(\ell,s) = 0, \quad
    \Psi'(0,s) = \Psi'(\ell,s) = 0,
\end{equation}
\begin{equation}\label{LD.IfC}
    \begin{aligned}
        W(\ell_0-0,s) = W(\ell_0+0,s), \quad \Psi(\ell_0-0,s) = \Psi(\ell_0+0,s), \quad \Psi'(\ell_0-0,s) = \Psi'(\ell_0+0,s), \\
        K\,\big(W'(\ell_0-0,s) - W'(\ell_0+0,s)\big) + (ms^2 + \varkappa + d s)\, W(\ell_0,s) = U(s).
    \end{aligned}
\end{equation}

\begin{remark}\label{st:ress}
If $s$ belongs to the resolvent set of the differential operator corresponding to problem~\eqref{TB.eqs}--\eqref{TB.IfC2}, then the boundary-value problem~\eqref{LD.sys}--\eqref{LD.IfC} has a unique solution.
\end{remark}

\begin{proof}[Sketch of proof]
    PDE problem~\eqref{TB.eqs}--\eqref{TB.IfC2} can be represented in the operator form
    \begin{equation*}
        \dot\xi={\rm A}\xi
    \end{equation*}
    in an appropriately chosen Hilbert space ${\cal H}$. Here $\xi=\begin{pmatrix} u, & v, & f, & g, & p, & q \end{pmatrix}^\top$ is the state vector with components
    $u(x)=w(\cdot,t)$, $v(x)=\dot w(\cdot,t)$, $f(x)=\psi(\cdot,t)$, $g(x)=\dot\psi(\cdot,t)$, $p=w(\ell_0,t)$, $q=\dot w(\ell_0,t)$.
    A second-order linear differential operator ${\rm A}$ establishes~\eqref{TB.eqs},~\eqref{TB.IfC2} and its domain incorporates conditions~\eqref{TB.BC}--\eqref{TB.IfC1}.
    
    Applying the Laplace transform to the state equation under zero initial conditions yields
    \begin{equation}\label{ressLT}
        ({\rm A}-s{\rm I})\hat\xi=0,
    \end{equation}
    providing it with corresponding boundary conditions. Here $\hat\xi(s)$ denotes the Laplace transform of $\xi(t)$.
    If $s$ belongs to the resolvent set of ${\rm A}$, then the operator ${\rm A}-s{\rm I}$ is invertible, implying that the corresponding boundary-value problem admits a unique solution.
\end{proof}

In what follows, we present a method of deriving the transfer function for system~\eqref{LD.sys}--\eqref{LD.IfC}.

Denote $W_0=W(x,s)$, $W_1=W'(x,s)$, $\Psi_0=\Psi(x,s)$, $\Psi_1=\Psi'(x,s)$, $\bar W = \begin{pmatrix} W_0, & W_1, & \Psi_0, & \Psi_1 \end{pmatrix}^\top$, and rewrite system~\eqref{LD.sys} in matrix form:
\begin{equation}\label{LD.sys.gen}
    \frac{\dd}{\dd x} \bar W = {\cal A} \bar W
\end{equation}
with
\begin{equation*}
    {\cal A}(s) = \left(\begin{array}{cccc}
        0 & 1 & 0 & 0 \\
        a & 0 & 0 & 1 \\
        0 & 0 & 0 & 1 \\
        0 & b & c & 0
    \end{array}\right), \quad
a(s) = \frac{\rho s^2}{K}, \quad b(s) = -\frac{K}{EI}, \quad c(s) = \frac{I_\rho s^2}{EI}-b(s).
\end{equation*}

Here and further, we will use the notation ${\cal E}^x$ for the matrix exponential $e^{x{\cal A}}(s)$. The solution of~\eqref{LD.sys.gen} is expressed in the following form:
\begin{equation*}
    \bar W(x,s) = \left\{\begin{array}{ll}
        {\cal E}^x\, \bar W^0,           & x\in[0,\ell_0], \\
        {\cal E}^{x-\ell}\, \bar W^\ell, & x\in(\ell_0,\ell],
    \end{array}\right.
\end{equation*}
where $\bar W^0 = \bar W(0,s) = \begin{pmatrix} W_0^0, & W_1^0, & \Psi_0^0, & \Psi_1^0 \end{pmatrix}^\top$, $\bar W^\ell = \bar W(\ell,s) = \begin{pmatrix} W_0^\ell, & W_1^\ell, & \Psi_0^\ell, & \Psi_1^\ell \end{pmatrix}^\top$.
Boundary conditions are $W_0^0=W_0^\ell=0$, $\Psi_1^0=\Psi_1^\ell=0$, and interface conditions result in nonhomogeneous algebraic system
\begin{equation}\label{LD.IfC.sys.gen}
    {\cal M}\begin{pmatrix} W_1^0, & \Psi_0^0, & W_1^\ell, & \Psi_0^\ell \end{pmatrix}^\top = \begin{pmatrix} 0, & 0, & 0, & U(s) \end{pmatrix}^\top
\end{equation}
with matrix
\begin{equation}\label{LD.IfC.sysM.gen}
    {\cal M}(s) = \left(\begin{array}{cccc}
        {\cal E}_{12}^{\ell_0} & {\cal E}_{13}^{\ell_0} & -{\cal E}_{12}^{\ell_0-\ell}  & -{\cal E}_{13}^{\ell_0-\ell} \\
        {\cal E}_{32}^{\ell_0} & {\cal E}_{33}^{\ell_0} & -{\cal E}_{32}^{\ell_0-\ell}  & -{\cal E}_{33}^{\ell_0-\ell} \\
        {\cal E}_{42}^{\ell_0} & {\cal E}_{43}^{\ell_0} & -{\cal E}_{42}^{\ell_0-\ell}  & -{\cal E}_{43}^{\ell_0-\ell} \\
        {\cal M}_{41}          & {\cal M}_{42}          & -K{\cal E}_{22}^{\ell_0-\ell} & -K{\cal E}_{23}^{\ell_0-\ell}
    \end{array}\right),
\end{equation}
where ${\cal M}_{41} = K{\cal E}_{22}^{\ell_0}+(ms^2+\varkappa+ds){\cal E}_{12}^{\ell_0}$,\; ${\cal M}_{42} = K{\cal E}_{23}^{\ell_0}+(ms^2+\varkappa+ds){\cal E}_{13}^{\ell_0}$.

The matrix ${\cal M}(s)$ arises from enforcing the interface conditions on the general solution of the spatial ODE system~\eqref{ressLT}.
If $s$ belongs to the resolvent set of the operator ${\rm A}$, then according to Remark~\ref{st:ress}, the boundary-value problem~\eqref{TB.eqs}--\eqref{TB.IfC2} admits a unique solution.
Consequently, the nonhomogeneous algebraic system~\eqref{LD.IfC.sys.gen} has a unique solution, and therefore the matrix ${\cal M}$ is nonsingular.

Let ${\cal M}^{-1}$ be the matrix inverse to ${\cal M}$. Then the unique solution of boundary-value problem~\eqref{LD.sys}--\eqref{LD.IfC} is presented as
\begin{equation*}
    W(x,s) = \left\{\begin{array}{ll}
        {\cal E}^x_{12} {\cal M}^{-1}_{14} U + {\cal E}^x_{13} {\cal M}^{-1}_{24} U,               & x\in[0,\ell_0], \\
        {\cal E}^{x-\ell}_{12} {\cal M}^{-1}_{34} U + {\cal E}^{x-\ell}_{13} {\cal M}^{-1}_{44} U, & x\in(\ell_0,\ell],
    \end{array}\right.
\end{equation*}
\begin{equation*}
    \Psi(x,s) = \left\{\begin{array}{ll}
        {\cal E}^x_{32} {\cal M}^{-1}_{14} U + {\cal E}^x_{33} {\cal M}^{-1}_{24} U,               & x\in[0,\ell_0], \\
        {\cal E}^{x-\ell}_{32} {\cal M}^{-1}_{34} U + {\cal E}^{x-\ell}_{33} {\cal M}^{-1}_{44} U, & x\in(\ell_0,\ell].
    \end{array}\right.
\end{equation*}

The transfer function is calculated as the output-to-input relation, $H(s) = \frac{Y(s)}{U(s)}$.
Taking the output signal in the form $Y_1(s)=W(\ell_k,s)$, or $Y_2(s)=\Psi'(\ell_k,s)$, we obtain the frequency-domain representation of infinite-dimensional system as proposed below.

\begin{proposition}\label{prp:TF.gen}
    Let $S=\{s\in{\mathbb C}\mid \det(M(s))\neq0\}$.
    Then, the transfer functions $H_1(s)$ and $H_2(s)$ of the control system~\eqref{TB.eqs}--\eqref{TB.IfC2}
    with the outputs $y_1(t) = w(\ell_k,t)$ and $y_2(t) = \psi'(\ell_k,t)$, respectively, are defined for all $s\in S$ as follows:
    $$H_1(s) = \left\{\begin{array}{ll}
        {\cal E}_{12}^{\ell_k}\, {\cal M}^{-1}_{14} + {\cal E}_{13}^{\ell_k}\, {\cal M}^{-1}_{24},           &\;\; \ell_k\in[0,\ell_0], \\
        {\cal E}_{12}^{\ell_k-\ell}\, {\cal M}^{-1}_{34} + {\cal E}_{13}^{\ell_k-\ell}\, {\cal M}^{-1}_{44}, &\;\; \ell_k\in(\ell_0,\ell],
        \end{array}\right.$$
    $$H_2(s) = \left\{\begin{array}{ll}
        {\cal E}_{42}^{\ell_k}\, {\cal M}^{-1}_{14} + {\cal E}_{43}^{\ell_k}\, {\cal M}^{-1}_{24},           &\;\; \ell_k\in[0,\ell_0], \\
        {\cal E}_{42}^{\ell_k-\ell}\, {\cal M}^{-1}_{34} + {\cal E}_{43}^{\ell_k-\ell}\, {\cal M}^{-1}_{44}, &\;\; \ell_k\in(\ell_0,\ell].
        \end{array}\right.$$
    Here, ${\cal M}_{j4}^{-1}$, $j=\overline{1,4}$, are the entries of the matrix inverse to~\eqref{LD.IfC.sysM.gen}.
\end{proposition}

\subsection{Poles on the Imaginary Axis}

Although the matrix~${\cal E}^x$ can, in general, be evaluated numerically, its entries can be expressed analytically in several important special cases. For example, if $s=0$, the matrix $x{\cal A}$ becomes nilpotent of index $4$, which yields the following matrix expression:
\begin{equation*}
    {\cal E}^x = \left(\begin{array}{cccc}
        1 & x+\frac{bx^3}6 & -\frac{bx^3}6  & \frac{x^2}2 \\
        0 & 1+\frac{bx^2}2 & -\frac{bx^2}2  & x \\
        0 & \frac{bx^2}2   & 1-\frac{bx^2}2 & x \\
        0 & bx             & -bx            & 1
    \end{array}\right)
\end{equation*}
and the boundary matrix for $s=0$ is as follows:
\begin{equation*}
    {\cal M} = \left(\begin{array}{cccc}
        \ell_0(1+\frac{b}{6}\ell_0^2)                                   & -\frac{b}{6}\ell_0^3                               & -(\ell_0-\ell)(1+\frac{b}{6}(\ell_0-\ell)^2) & \frac{b}{6}(\ell_0-\ell)^3 \\
        \frac{b}{2}\ell_0^2                                             & 1-\frac{b}{2}\ell_0^2                              & -\frac{b}{2}(\ell_0-\ell)^2                  & -1+\frac{b}{2}(\ell_0-\ell)^2 \\
        b\ell_0                                                         & -b \ell_0                                          & -b(\ell_0-\ell)                              & b(\ell_0-\ell) \\
        K(1+\frac{b}{2}\ell_0^2)+\varkappa\ell_0(1+\frac{b}{6}\ell_0^2) & -K\frac{b}{2}\ell_0^2-\varkappa\frac{b}{6}\ell_0^3 & -K(1+\frac{b}{2}(\ell_0-\ell)^2)             & K\frac{b}{2}(\ell_0-\ell)^2 
    \end{array}\right),
\end{equation*}
${\rm rank}({\cal M})=4$; \;
$\det({\cal M}) = -\frac{b\ell}3 \Big( 3K\ell - \varkappa\ell_0(\ell_0-\ell) \big( b\ell_0(\ell_0-\ell)+3 \big) \Big)$.

As mentioned in Section~\ref{sec:Intro}, the underlined beam-body system possesses two purely imaginary eigenvalues $\pm i\omega_c$, where $\omega_c = \sqrt\frac{K}{I_\rho}$ is the cutoff frequency.
This observation implies so called marginal stability of the considered locally damped system.
Note that for a particular mechanical setup with configuration~\eqref{mechpar} these eigenvalues $s \approx \pm 3.2394\cdot10^6\,i$ correspond to physical frequency approximately $0.52$~MHz.

\begin{statement}\label{re:ImEigs}
    The transfer functions $H_1(s)$ and $H_2(s)$ from Proposition~\ref{prp:TF.gen} possess poles at $s=\pm i\omega_c$.
\end{statement}

\begin{proof}
    For resolving~\eqref{LD.sys} in the cases $s=\pm i\sqrt{K/I_\rho}$, we denote $\zeta_1=\sqrt{KI_\rho+\rho EI}$, $\zeta_2=\sqrt{I_\rho EI}$ and $\tilde x=\zeta_1/\zeta_2\, x$. Then
    \begin{equation*}
        {\cal E}^x = \frac1{\zeta_1^2} \left(\begin{array}{cccc}
            KI_\rho+\rho EI\cos\tilde x                  & \zeta_1\zeta_2\sin\tilde x                  & 0         & -\zeta_2^2(\cos\tilde x-1) \\
            -\rho EI\frac{\zeta_1}{\zeta_2}\sin\tilde x  & \zeta_1^2\cos\tilde x                       & 0         & \zeta_1\zeta_2\sin\tilde x \\
            K\rho(x-\frac{\zeta_2}{\zeta_1}\sin\tilde x) & KI_\rho(\cos\tilde x-1)                     & \zeta_1^2 & \rho EIx+KI_\rho\frac{\zeta_2}{\zeta_1}\sin\tilde x \\
            -K\rho(\cos\tilde x-1)                       & -KI_\rho\frac{\zeta_1}{\zeta_2}\sin\tilde x & 0         & \rho EI+KI_\rho\cos\tilde x
        \end{array}\right).
    \end{equation*}
    Since the third column does not depend on $x$, the second and fourth columns $\begin{pmatrix} 0, & \pm1, & 0, & 0 \end{pmatrix}^\top$ of the boundary matrix ${\cal M}$ are linearly dependant, therefore $\det({\cal M})=0$.
\end{proof}

\section{Reduced-Order Modeling and Numerical Analysis}\label{sec:rom}
In this section, we present the results of application of the Loewner matrix framework to the considered model of a hybrid system.
We rely on the reduction procedure described in~\cite{ZG2023} for the Euler–Bernoulli beam model with structural damping.

Let us fix a natural number $P$ and a frequency $\nu_{max}$, define the interval $[-2\pi i \nu_{max};\, 2\pi i \nu_{max}]$ on the imaginary axis and partition it into nodes $\omega_l$, where $l=\overline{1,2P}$, thus generating a sequence of uniformly spaced points. In doing so, we ensure that the point $0$ is not included into this partition.

We assume that experimental measurements are given as a set of values of the transfer function $H(s)$ at the specified nodes: let $\mu_i=\omega_{2i-1}$, $\lambda_i=\omega_{2i}$, $i=\overline{1,P}$, and suppose that values $v_i=H(\mu_i)$, $w_i=H(\lambda_i)$ are known and the sets
$\left\{ (\lambda_j, w_j) \right\}_{j=1}^P$ and $\left\{ (\mu_i, v_i) \right\}_{i=1}^P$
are disjoint.
The scheme of data splitting used here is referred to as alternate splitting, see, e.g.,~\cite{KGA2021} where other types of sampling nodes arrangement are analyzed.

The goal is to construct a rational function $\tilde H(s)$ whose values coincide with the given values of $H(s)$ at the corresponding nodes, i.e., $\tilde H(\mu_i)=v_i$, $\tilde H(\lambda_j)=w_j$. To do this, we implement the representation using Loewner matrices, that in this case are complex-valued matrices of dimension $P\times P$ with elements of the form
\begin{equation*}
    \LL_{ij} = \frac{v_i-w_j}{\mu_i-\lambda_j}, \quad \MM_{ij} = \frac{\mu_i v_i-\lambda_j w_j}{\mu_i-\lambda_j}.
\end{equation*}
In addition, a real-valued column vector $\VV$ and a real-valued row vector $\WW$ of dimension $P$ are composed of the elements $v_i$ and $w_j$, respectively.
Next, we determine the order $r$ of the reduced system by applying the singular value decomposition (SVD) of the matrix pencil:
$[\LL, \MM] = \bY \Sigma_1 \bar\bX$, where $\bY \in {\mathbb C}^{P\times P}$, $\bar\bX \in {\mathbb C}^{2P\times 2P}$, and $\Sigma_1$ is a $P\times 2P$-matrix of the real singular values of the pencil, sorted decreasingly. Similarly, from the decomposition
$\left[\begin{array}{c} \LL \\[-2pt] \MM \end{array}\right] = \bar\bY \Sigma_2 \bX$ we obtain matrices
$\bar\bY \in {\mathbb C}^{2P\times 2P}$, $\bX \in {\mathbb C}^{P\times P}$ and $\Sigma_2 \in {\mathbb R}^{2P\times P}$.

Computations confirm that it is possible to choose a number $r$ such that, starting from this number, the rate of decrease of the singular values of the pencil becomes small compared to the difference between adjacent values with numbers less than $r$, while the divergence between $\Sigma_1$ and $\Sigma_2$ is negligible.
Since the original system is infinite-dimensional, usually all singular values are non-zero.
In a good scenario, we observe the rapid decrease of singular values of $[\LL, \MM]$ up to some number, after which all singular values are close to zero, whereas in other situations the picture becomes less sharp, although it is still possible to estimate the number of key degrees of freedom of the system by comparing approximation errors for the estimated range of $r$.
Note that there is a risk of overfitting when $r$ is too large.
At the end, the choice of the optimal in some sense $r$ is a trade-off between accuracy and large-scaling and it greatly relies on the experimentation.
In this work, we do not define strict criterion of the optimal choice of the reduction order $r$, instead we calculate and compare the relative and maximal errors of the interpolation for the estimated range of acceptable values of $r$.
Next, we construct the reduced system of order $r$. For this purpose, we form $\bY_r$ and $\bX_r$~--- submatrices of dimensions $P\times r$ and $r\times P$ of the matrices $\bY$ and $\bX$, respectively, and then compute matrices of the form
$$\bE = -\bX_r \LL \bY_r, \quad \bA = -\bX_r \MM \bY_r, \quad \bB = \bX_r \VV, \quad \bC = \WW \bY_r.$$
The desired function $\tilde H(s) = \bC (s \bE - \bA)^{-1} \bB$ is an approximation of the transfer function $H(s)$ of the original infinite-dimensional system.

For the numeric simulations, we use the following values of mechanical parameters:
\begin{equation}\label{mechpar}
    \begin{aligned}
        & \ell=1.89~{\rm m}, \; \ell_0=1.378~{\rm m}, \; A=2.25~{\rm cm^2}, \; m=0.1~{\rm kg}, \; \varkappa=7~{\rm N/mm}, \\
        & \rho_0=2700~{\rm kg/m^3}, \; E=69~{\rm GPa}, \; G=25.5~{\rm GPa}, \; I=1.6875\cdot10^{-10}~{\rm m^4}.
    \end{aligned}
\end{equation}
The sampling range here is determined by $\nu_{max}=250$~Hz, and $P=1000$, so that the corresponding segment of the imaginary axis is partitioned by $2000$ points.

In what follows, we present graphic results: Figure~\ref{fig:co_y1_svd} demonstrates the decay of singular values of pencil $[\LL, \MM]$ for the system with collocated actuator-sensor pair, i.e. $\ell_k=\ell_0$, displacement-type output, under different choices of the damping coefficient.
Figures~\ref{fig:co_d025}--\ref{fig:nc_y2}, obtained for systems with different input-output and damping configurations specified in each figure's title, depict magnitude Bode plots of the original transfer function $H(s)$ and its Loewner approximation $\tilde H(s)$ of order $r$, $s=2i\pi\nu$, $\nu\in[0, \nu_{max}]$.
Unlike in Proposition~\ref{prp:TF.gen}, we omit subscripts here, using the notation $H$ to denote the function $H_1$ for the case of input-output collocation, and $H_2$ for non-collocation, specifying the corresponding configuration in the figure title.
When considering the non-collocation, we take $\ell_k=0.7325$~m.

Additionally, we display 
the maximal approximation error $E_{max} = \max|H(s)-\tilde H(s)|$ over the sampled frequency range, which is treated as the discrete analogue of the $H^\infty$-norm.

\begin{figure}[!h]
    \begin{minipage}[b]{\linewidth} \centering
        \begin{subfigure}{0.49\linewidth} \centering
            \includegraphics[width=\linewidth]{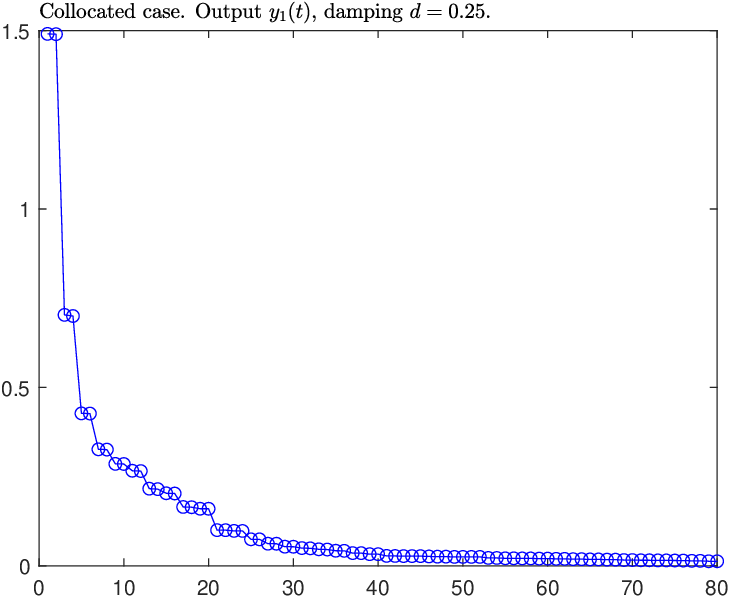}
            \caption{$d=0.25$}\label{fig:co_y1_d025_svd}
        \end{subfigure} \hfill
        \begin{subfigure}{0.49\linewidth} \centering
            \includegraphics[width=\linewidth]{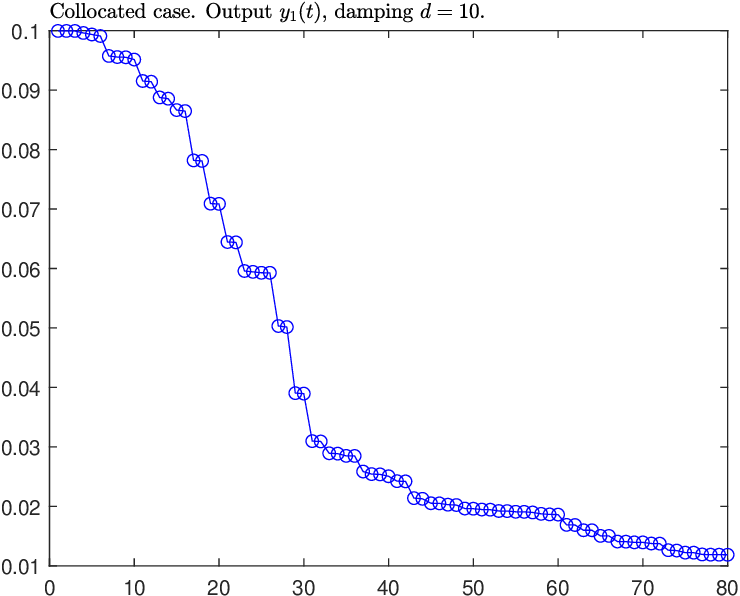}
            \caption{$d=10$}\label{fig:co_y1_d10_svd}
        \end{subfigure}
    \end{minipage} \vspace{-1ex}
    \caption{Singular value decay. Collocated configuration with output $y_1(t)$: $d=0.25$ (left) and $d=10$ (right).}\label{fig:co_y1_svd}
\end{figure}

\begin{figure}[!h]
    \begin{minipage}[b]{\linewidth} \centering
        \begin{subfigure}{0.49\linewidth} \centering
            \includegraphics[width=\linewidth]{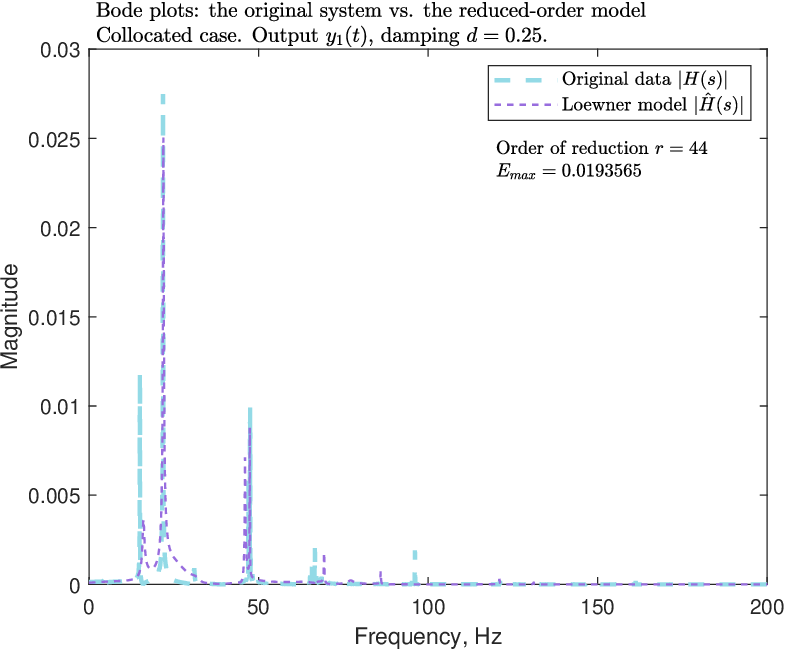}
            \caption{Output $y_1(t)$}\label{fig:co_y1_d025_r44}
        \end{subfigure} \hfill
        \begin{subfigure}{0.49\linewidth} \centering
            \includegraphics[width=\linewidth]{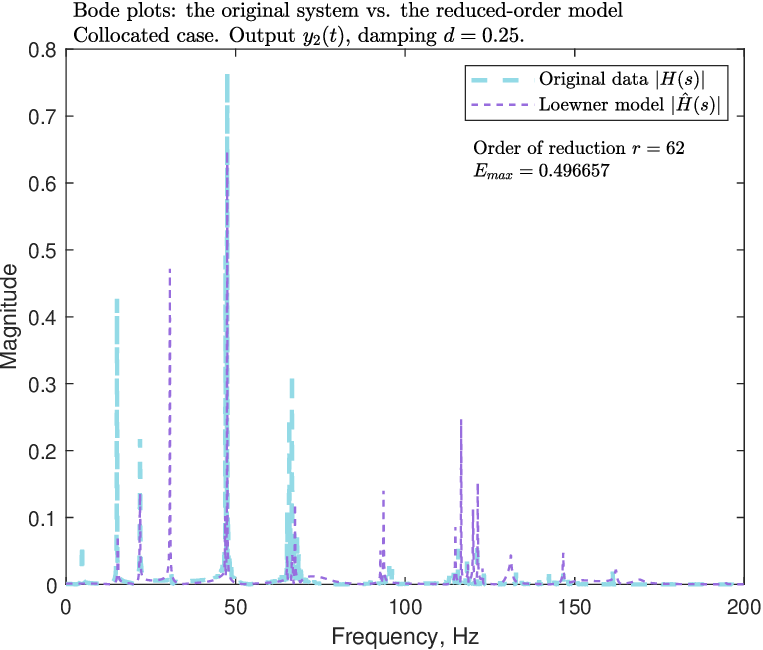}
            \caption{Output $y_2(t)$}\label{fig:co_y2_d025_r62}
        \end{subfigure}
    \end{minipage} \vspace{-1ex}
    \caption{Bode plots for the collocated configuration with outputs $y_1(t)$ (left) and $y_2(t)$ (right).}\label{fig:co_d025}
\end{figure}

\begin{figure}[!h]
    \begin{minipage}[b]{\linewidth} \centering
        \begin{subfigure}{0.49\linewidth} \centering
            \includegraphics[width=\linewidth]{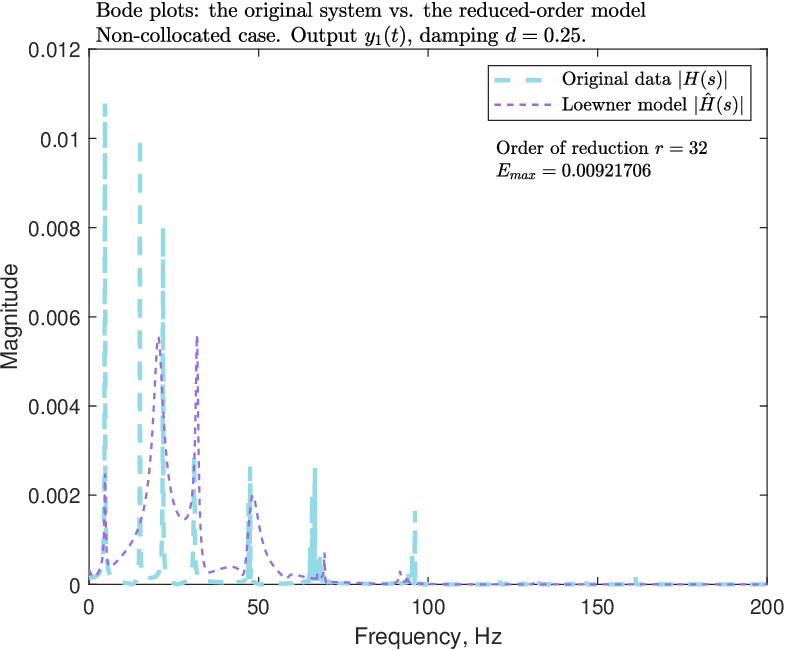}
            \caption{$d=0.25$}\label{fig:nc_y1_d025_r32}
        \end{subfigure} \hfill
        \begin{subfigure}{0.49\linewidth} \centering
            \includegraphics[width=\linewidth]{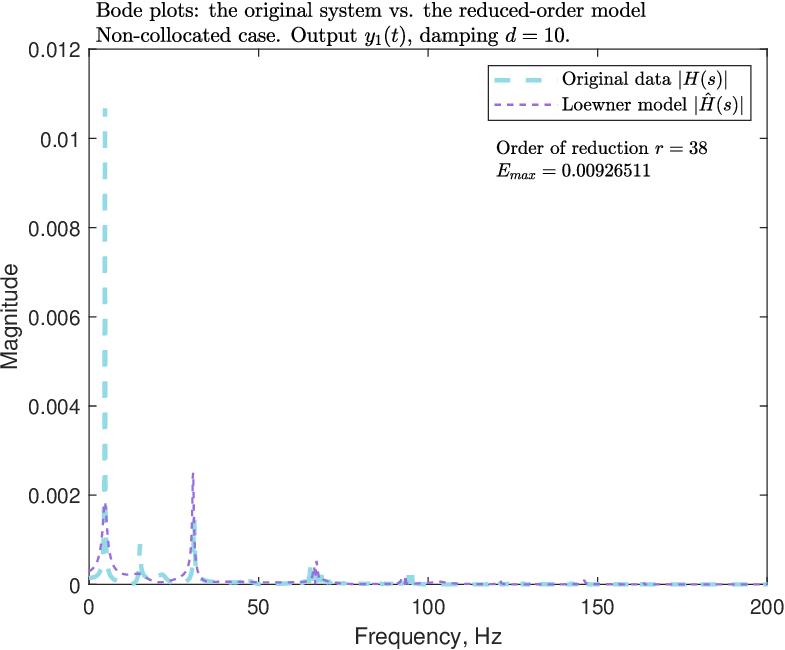}
            \caption{$d=10$}\label{fig:nc_y1_d10_r38}
        \end{subfigure}
    \end{minipage} \vspace{-1ex}
    \caption{Bode plots for the non-collocated configuration with output $y_1(t)$; $d=0.25$ (left) and $d=10$ (right).}\label{fig:nc_y1}
\end{figure}

\begin{figure}[!h]
    \begin{minipage}[b]{\linewidth} \centering
        \begin{subfigure}{0.49\linewidth} \centering
            \includegraphics[width=\linewidth]{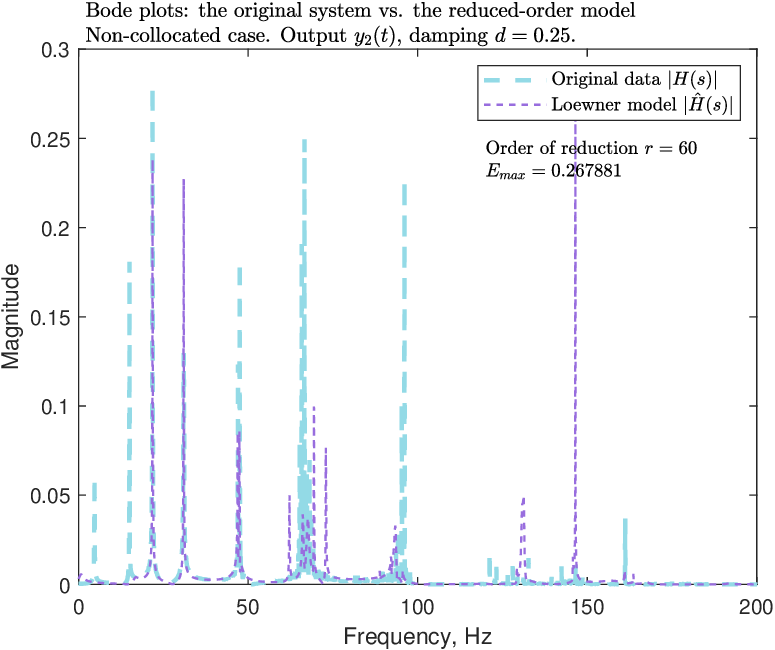}
            \caption{$d=0.25$}\label{fig:nc_y2_d025_r60}
        \end{subfigure} \hfill
        \begin{subfigure}{0.49\linewidth} \centering
            \includegraphics[width=\linewidth]{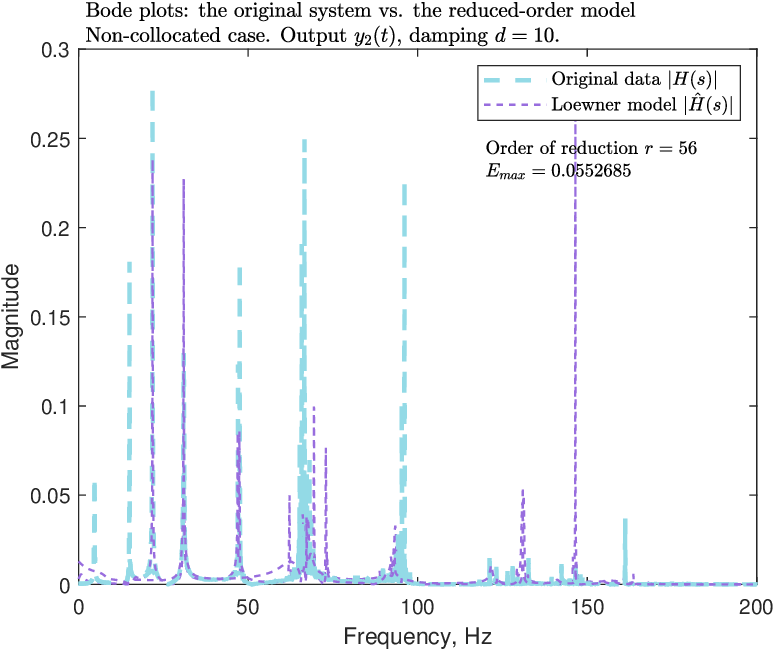}
            \caption{$d=10$}\label{fig:nc_y2_d10_r56}
        \end{subfigure}
    \end{minipage} \vspace{-1ex}
    \caption{Bode plots for the non-collocated configuration with output $y_2(t)$: $d=0.25$ (left) and $d=10$ (right).}\label{fig:nc_y2}
\end{figure}

The obtained graphs demonstrate good agreement between the original and reduced-order models in terms of the transfer function behavior on the imaginary axis, particularly in the vicinity of the poles.

\section{Conclusion}

The behavior of the system at the cutoff frequency described in Statement~\ref{re:ImEigs}, with account of the aforementioned dissipative property of the system, apparently indicates the following.
Although the passive stabilizer, implemented as a point damper through velocity feedback on the transverse displacement at the beam–body interface, shifts most of the spectrum into the left half-plane, the cutoff frequency remains strictly on the imaginary axis. In fact, $\omega_c$ always belongs to the spectrum of a simply supported single-span beam, as shown in~\cite{Cazzani2016}.
Consequently, the damped beam achieves marginal rather than asymptotic stability.
The eigenfunction corresponding to $s=\pm i\omega_c$ is
$\begin{pmatrix} W(x) \\ \Psi(x) \end{pmatrix} = \begin{pmatrix} 0 \\ \Psi(0) \end{pmatrix}$, and $\Psi(0)=\Psi(\ell)$,
i.e. the transverse deflection vanishes while the cross-sectional rotation settles into a spatially uniform constant.
This pure-shear vibration mode of Timoshenko beams have been first studied in~\cite{Downs1976}.
It is reasonable to envisage that as $t \to \infty$, the trajectory converges to a non-trivial limit manifold 
$\left\{ (u, v, f, g, p, q) \in {\cal H} \mid u(\cdot,t)\equiv 0,\: v(\cdot,t)\equiv 0,\: f(\cdot,t)\equiv C,\: g(\cdot,t)\equiv 0,\: p=0,\: q=0 \right\}$,
where $C \in \mathbb{R}$ is a constant determined by the initial data. The notation here corresponds to Remark~\ref{st:ress}.
Geometrically, this manifold defines a non-trivial steady state of pure static shear where the beam's unperturbed rectangular profile permanently deforms into a parallelogram.
This behavior, apparently, signifies partial-state asymptotic stability, where the kinetic energy dissipates and the system is stabilized with respect to the velocities and displacement, permanently retaining a residual static shear deformation.
From the physical point of view, such behavior of the mechanical system is justified by the design of the damping mechanism~--- the feedback affects directly transversal deflections but not the angle of rotation of cross-sections. Since the mode of oscillations corresponding to the cutoff frequency does not induce deflections, it becomes ``invisible'' to the stabilizer, thus remaining undamped.
A rigorous proof of the discussed properties requires a non-trivial spectral analysis of the corresponding differential operator. In particular, we consider the problem of partial-state stabilization as a promising direction for the authors' future investigation. Note that the Euler--Bernoulli beam does not demonstrate the discussed effect.

In this work, we studied the effect of localized dissipation, whereas the implementation of distributed damping models~\cite{ADW1986, LZ2018} represents a promising direction for future research.
Moreover, the analysis of alternative data-splitting schemes and their effects on the approximation accuracy of Loewner interpolants~\cite{KGA2021} is another prospective direction for studying the class of mechanical systems considered here.

\section*{Acknowledgment}
The second author was supported by a grant from the Simons Foundation (SFI-PD-Ukraine-00017674, Julia Kalosha).

\bibliographystyle{aomplain}
\bibliography{Bib}

\end{document}